\documentclass[12pt,a4paper,psamsfonts]{amsart}
\usepackage{fancyhdr}
\usepackage{appendix}
\usepackage{amssymb,amscd,amsxtra,calc}
\usepackage{mathrsfs}
\usepackage{cmmib57}
\usepackage{multirow}
\usepackage[all]{xy}
\usepackage[colorlinks,linkcolor=blue,anchorcolor=blue,citecolor=green]{hyperref}
\theoremstyle{plain}
    \newtheorem{thm}{Theorem}[section]

    \newtheorem{corollary}[thm]{Corollary}
    \newtheorem{example}[thm]{Example}
    \newtheorem{lemma}[thm]{Lemma}
    \newtheorem{proposition}[thm]{Proposition}
    \newtheorem{question}[thm]{Question}
    \newtheorem{theorem}[thm]{Theorem}

\theoremstyle{definition}
    \newtheorem{definition}[thm]{Definition}

    \newtheorem{remark}[thm]{Remark}
\theoremstyle{remark}

\newcommand{\C}{\mathbb{C}}

\newcommand{\K}{\mathbb{K}}

\newcommand{\Q}{\mathbb{Q}}

\newcommand{\R}{\mathbb{R}}

\newcommand{\Z}{\mathbb{Z}}

\newcommand{\Aut}{\operatorname{Aut}}

\newcommand{\ch}{\operatorname{char }}

\newcommand{\diag}{\operatorname{diag}}
\newcommand{\End}{\operatorname{End}}
\newcommand{\Exc}{\operatorname{Exc}}

\newcommand{\IAmp}{\operatorname{IAmp}}
\newcommand{\id}{\operatorname{id}}

\newcommand{\Ker}{\operatorname{Ker}}
\newcommand{\NE}{\overline{\operatorname{NE}}}

\newcommand{\NS}{\operatorname{NS}}

\newcommand{\Pol}{\operatorname{Pol}}

\newcommand{\SEnd}{\operatorname{SEnd}}

\newcommand{\Sing}{\operatorname{Sing}}

\newcommand{\variety}{\operatorname{variety}}

\newcommand{\Per}{\operatorname{Per}}

\newcommand{\N}{\operatorname{N}}

\newcommand{\Alb}{\operatorname{Alb}}

\newcommand{\alg}{\mathrm{alg}}
\newcommand{\reg}{\mathrm{reg}}

\begin{document}

\title[Semi-group structure of all endomorphisms]
{Semi-group structure of all endomorphisms of a projective variety admitting a polarized endomorphism}

\author{Sheng Meng}

\address
{
\textsc{Department of Mathematics} \endgraf
\textsc{National University of Singapore,
Singapore 119076, Republic of Singapore
}}
\email{math1103@outlook.com}
\email{ms@u.nus.edu}

\author{De-Qi Zhang}

\address
{
\textsc{Department of Mathematics} \endgraf
\textsc{National University of Singapore,
Singapore 119076, Republic of Singapore
}}
\email{matzdq@nus.edu.sg}

\begin{abstract}
Let $X$ be a projective variety admitting a polarized (or more generally, int-amplified) endomorphism.
We show: there are only finitely many contractible extremal rays; and when $X$ is $\Q$-factorial normal, every minimal model program is equivariant relative to the monoid $\SEnd(X)$ of all surjective endomorphisms, up to finite index.

Further, when $X$ is rationally connected and smooth, we show:
there is a finite-index submonoid $G$ of $\SEnd(X)$ such that $G$
acts via pullback as diagonal (and hence commutative)
matrices on the Neron-Severi group;
the full automorphisms group $\Aut(X)$ has finitely many connected components;
and every amplified endomorphism is int-amplified.
\end{abstract}



\maketitle
\tableofcontents

\section{Introduction}
We work over an algebraically closed field $k$ which has characteristic zero (unless otherwise indicated),
and is uncountable (only used to guarantee the birational invariance of the rational connectedness property).
Let $f$ be a surjective endomorphism of a projective variety $X$.
We say that $f$ is $q$-{\it polarized} if $f^*L\sim qL$ (linear equivalence) for some ample Cartier divisor $L$ and integer $q>1$.
We say
that $f$ is {\it amplified} (resp. {\it int-amplified}),
if $f^*L-L$ is ample for some Cartier (resp.~ample Cartier) divisor $L$.
The notion of amplified endomorphisms $f$ was first defined by Krieger and Reschke (cf.~\cite{Kr}).
Fakhruddin showed that
such $f$ has a countable Zariski-dense subset of periodic points (cf. \cite[Theorem 5.1]{Fak}).

We refer to \cite{KM} for the definitions of log canonical (lc), klt or terminal singularities.
A sequence $X = X_1 \dasharrow X_2 \dasharrow \cdots$ of log MMP (= {\it log minimal model program}) consists of
the contraction $X_i \dasharrow X_{i+1}$ of a $(K_{X_i} + \Delta_i)$-negative extremal ray
for some log canonical pair $(X_i, \Delta_i)$
with $\Delta_i$ being an $\R$-Cartier effective divisor, which is of divisorial, flip or Fano type.
An MMP is a log MMP with $\Delta = 0$.
We refer to \cite[Theorem 1.1]{Fu11} for the cone theorem of lc pairs and \cite[Corollary 1.2]{Bi} for the existence of lc flips.

A submonoid $G$ of a monoid $\Gamma$ is said to be of {\it finite-index} in $\Gamma$
if there is a chain $G = G_0 \le G_1 \le \cdots \le G_r = \Gamma$ of submonoids and homomorphisms $\rho_i : G_i \to F_i$ such that $\Ker(\rho_i) = G_{i-1}$ and all $F_i$ are finite
{\it groups}.

Theorem \ref{main-thm-finite-R} below is the most crucial result of the paper.
First, an lc pair $(X, \Delta)$ may have infinitely many $(K_X + \Delta)$-negative extremal rays.
Theorem \ref{main-thm-finite-R} below implies that this case will never happen if we assume $X$ admits a polarized (or int-amplified) endomorphism (see also Theorem \ref{prop-finite-cont-r} for a more general result).

Let $\SEnd(X)$ be the set of {\it all surjective endomorphisms} on $X$.

Theorem \ref{main-thm-finite-R} below says that every finite sequence of MMP starting from a $\Q$-factorial normal $X$ is equivariant (cf.~Definition \ref{def-G-equiv}) relative to $\SEnd(X)$, up to finite-index.
Note that $\SEnd(X)$ is usually a huge infinite set; and also the image $h_*(R)$ of a $K_X$-negative extremal
ray $R$ in the closed effective $1$-cycle cone $\NE(X)$ under a map $h \in \SEnd(X)$, may not be
$K_X$-negative anymore. So we have to deal with general (not necessarily $K_X$-negative) contractible
extremal rays $R$ of $\NE(X)$ in the sense of Definition \ref{def:extrem_ray}.
Our Theorem \ref{main-thm-finite-R} below asserts the finiteness of these rays $R$,
without even assuming $K_X$ being $\Q$-Cartier.

\begin{theorem}\label{main-thm-finite-R}(cf.~Theorems \ref{prop-finite-cont-r} and \ref{thm-GMMP})
Let $X$ be a (not necessarily normal or $\Q$-Gorenstein) projective variety with a polarized (or int-amplified) endomorphism. Then:
\begin{itemize}
\item[(1)]
$X$ has only finitely many (not necessarily $K_X$-negative) contractible extremal rays in the sense of Definition \ref{def:extrem_ray}.
\item[(2)]
Suppose $X$ is $\Q$-factorial normal. Then any finite sequence of MMP starting from $X$ is $G$-equivariant for some finite-index submonoid $G$ of $\SEnd(X)$.
\end{itemize}
\end{theorem}

We extend the results in \cite{MZ} and \cite{Meng}
about equivariant MMP from being relative to a single polarized or int-amplified
endomorphism to the whole $\SEnd(X)$ up to finite-index.
When $X$ is a point, every endomorphism of $X$ is regarded as being polarized.
A normal projective variety $X$ is said to be {\it $Q$-abelian} if there is a finite surjective morphism $\pi:A\to X$ \'etale in codimension $1$ with $A$ being an abelian variety.

\begin{theorem}\label{main-thm-GMMP}
Let $f:X\to X$ be an int-amplified endomorphism of a $\mathbb{Q}$-factorial klt projective variety $X$.
Then there exist a finite-index submonoid $G$ of $\SEnd(X)$,
a $Q$-abelian variety $Y$, and a $G$-equivariant (cf.~ Definition \ref{def-G-equiv}) relative MMP over $Y$
$$X=X_0 \dashrightarrow \cdots \dashrightarrow X_i \dashrightarrow \cdots \dashrightarrow X_r=Y$$
(i.e. $g \in G = G_0$ descends to $g_i \in G_i$ on each $X_i$), with every $X_i \dashrightarrow X_{i+1}$
a divisorial contraction, a flip or a Fano contraction, of a $K_{X_i}$-negative extremal ray, such that:
\begin{itemize}
\item[(1)]
There is a finite quasi-\'etale Galois cover $A \to Y$ from an abelian variety $A$
such that $G_Y := G_r$ lifts to a submonoid $G_A$ of $\SEnd(A) \le \End_{\variety}(A)$.
\item[(2)]
If $g$ in $G$ is polarized (resp. int-amplified), then so are
its descending $g_i$ on $X_i$ and the lifting to $A$ of $g_r$ on $X_r = Y$.
\item[(3)]
If $g$ in $G$ is amplified and
its descending $g_i$ on $X_i$ is int-amplified for some $i$, then $g$ is int-amplified.
\item[(4)]
If $K_X$ is pseudo-effective, then $X=Y$ and it is $Q$-abelian.
\item[(5)]
If $K_X$ is not pseudo-effective, then for each $i$, $X_i\to Y$ is equi-dimensional holomorphic with every fibre (irreducible) rationally connected.
The $X_{r-1}\to X_r = Y$ is a Fano contraction.
\item[(6)]
For any subset $H \subseteq G$ and its descending $H_Y \subseteq \SEnd(Y)$,
$H$ acts via pullback on $\NS_{\Q}(X)$ or $\NS_{\C}(X)$ as commutative diagonal matrices with respect
to a suitable basis if and only if so does $H_Y$.
\end{itemize}
\end{theorem}

Let $\Pol(X)$ be the set of {\it all polarized endomorphisms} on $X$, and
let $\IAmp(X)$ be the set of {\it all int-amplified endomorphisms} on $X$.
In general, they are not semigroups, i.e., they may not be closed under composition; see Example \ref{ex:power_eq_map}.
When $X$ is rationally connected and smooth,
Theorem \ref{main-thm-rc} below gives the assertion that if $g$ and $h$ are in $\Pol(X)$ (resp.~$\IAmp(X)$)
then $g^M \circ h^M$ remains in $\Pol(X)$ (resp.~ $\IAmp(X)$) for some $M>0$ depending only on $X$.
For general $X$, Corollary \ref{main-cor-comp} says that the same assertion on $X$
is reduced to that on
the base of the end product $Y$ of the MMP starting from $X$, or
the quasi-\'etale abelian variety cover $A$ of $Y$.

\begin{corollary}\label{main-cor-comp}
We use the notation and assumption in Theorem \ref{main-thm-GMMP}. For $g, h$ in $G \subseteq \SEnd(X)$,
let $\tau = g \circ h$, $\tau_Y = g_Y \circ h_Y$ the descending to $Y$
and $\tau_A = g_A \circ h_A$ the lifting to $A$. Then we have:
\begin{itemize}

\item[(I)]
Suppose both $g^*, h^*$ are diagonalizable on $\NS_{\C}(X)$
(resp.~  both $g^*, h^*$ are diagonalizable on $\NS_{\Q}(X)$; both $g, h$ are in $\Pol(X)$; one of $g, h$ is in $\IAmp(X)$). Then
{\rm (Ia)} and {\rm (Ib)} below are equivalent.

\item[(Ia)]
$\tau^*$ is diagonalizable on $\NS_{\C}(X)$
(resp. $\tau^*$ is diagonalizable on $\NS_{\Q}(X)$; $\tau \in \Pol(X)$;  $\tau \in \IAmp(X)$).

\item[(Ib)]
$\tau_Y^*$ is diagonalizable on $\NS_{\C}(Y)$
(resp. $\tau_Y^*$ is diagonalizable on $\NS_{\Q}(Y)$; $\tau_Y \in \Pol(Y)$; $\tau_Y \in \IAmp(Y)$).

\item[(II)]
$\tau_Y$ is in $\Pol(Y)$ (resp. $\IAmp(Y)$) if and only if
$\tau_A$ is in $\Pol(A)$ (resp. $\IAmp(A)$).

\item[(III)]
Suppose that both $g^*$ and $h^*$ are diagonalizable on $\NS_{\C}(X)$.
Then $g^* h^* = h^* g^*$ on $\NS_{\C}(X)$ if and only if $g_Y^* h_Y^* = h_Y^* g_Y^*$ on $\NS_{\C}(Y)$.
\end{itemize}
\end{corollary}

By the results in \cite{MZ} and \cite{Meng}, we know that the building blocks of polarized (or more generally int-amplified) endomorphisms are those on
Abelian varieties and rationally connected varieties.
Indeed, if $X$ has mild singularities, is non-uniruled and admits a polarized (resp.~int-amplified) endomorphism, then
$X$ is a $Q$-abelian variety: there is a finite Galois cover $A \to X$ \'etale in codimension one
such that $f$ lifts to a polarized (resp.~int-amplified) endomorphism on the abelian variety $A$;
if $X$ is uniruled, then a polarized (resp. int-amplified) $f$ descends to a polarized (resp. int-amplified) endomorphism on the base $Y$
of a special maximal rationally connected fibration $X \dasharrow Y$, and $Y$ is non-uniruled,  hence
it is a $Q$-abelian variety; see \cite[Proposition 1.6]{MZ}, \cite[Corollary 4.20]{Na10}.
Therefore, the essential building blocks we have to study are those polarized (resp. int-amplifed) endomorphisms
on rationally connected varieties.

Our next Theorem \ref{main-thm-rc} gives the structure of the monoid $\SEnd(X)$ for a
rationally connected $X$.
The second assertion below says that the surjective endomorphisms on a rationally connected variety
admitting a polarized (or int-amplifed) endomorphism, act as diagonal (and hence commutative)
matrices on the Neron-Severi group,
up to finite-index.

Though $\Pol(X)$ and $\IAmp(X)$ may not be subsemigroups of $\SEnd(X)$,
the third and fourth assertions below say that they are semigroups ``up to finite-index";
it also answers affirmatively \cite[Question 4.15]{YZ}, ``up to finite-index", when $X$ is rationally connected and smooth.
By Example \ref{ex:power_eq_map}, this extra ``up to finite-index" assumption
is necessary.

The fourth assertion below also says that the pullback action of $\SEnd(X)$ on $\NS_{\Q}(X)$ is determined by
that of $\IAmp(X)$, up to finite-index ({\it hence the importance of studying int-amplified endomorphisms}).
For a subset $S$ of a semigroup $H$ and an integer $M \ge 1$, denote by
$\langle S^{[M]} \rangle :=
\{s_1^M \cdots s_r^M \, | \, r \ge 1, \, s_i \in S\}$.

\begin{theorem}\label{main-thm-rc} (cf.~Theorem \ref{thm-q-pi-pt})
Let $X$ be a rationally connected smooth projective variety admitting
a polarized (or int-amplified) endomorphism $f$.
We use the notation $X=X_0 \dashrightarrow \cdots \dashrightarrow X_r=Y$ and the finite-index submonoid $G \le \SEnd(X)$ as in Theorem \ref{main-thm-GMMP}.
Then there is an integer $M \ge 1$ depending only on $X$ such that:
\begin{itemize}
\item[(1)] The $Y$ in Theorem \ref{main-thm-GMMP} is a point.
\item[(2)]
$G^*|_{\NS_{\Q}(X)}$ is a commutative diagonal
monoid with respect to a suitable $\Q$-basis $B$ of $\NS_{\Q}(X)$.
Further, for every $g$ in $G$,
the representation matrix $[g^*|_{\NS_{\Q}(X)}]_B$ relative to $B$,
is equal to $\diag[q_1, q_2, \dots]$ with integers $q_i \ge 1$.
\item[(3)]
$G \cap \Pol(X)$ is a subsemigroup of $G$, and consists exactly of those $g$ in $G$
such that $[g^*|_{\NS_{\Q}(X)}]_B = \diag[q, \dots, q]$ for some integer $q \ge 2$.
Further,
$$G \cap \Pol(X) \supseteq \langle \Pol(X)^{[M]} \rangle .$$
\item[(4)]
$G \cap \IAmp(X)$ is a subsemigroup of $G$, and consists exactly of those $g$ in $G$
such that $[g^*|_{\NS_{\Q}(X)}]_B = \diag[q_1, q_2, \dots]$ with integers $q_i \ge 2$.
Further,
$$G \, (G \cap \IAmp(X)) = G \cap \IAmp(X)  \supseteq  \langle \IAmp(X)^{[M]} \rangle ;$$
any $h$ in $\SEnd(X)$ has $(h^M)^* = (g_1^*)^{-1} g_2^*$ on $\NS_{\Q}(X)$
for some $g_i$ in $G \cap \IAmp(X)$.
\item[(5)]
We have $h^M\in G$ and that $h^*|_{\NS_{\C}(X)}$ is diagonalizable for every $h \in \SEnd(X)$.
\end{itemize}
\end{theorem}

Let $\Aut(X)$ be the group of all automorphisms of $X$, and $\Aut_0(X)$ its neutral connected component.
By applying Theorem \ref{main-thm-rc}, we have the following result.

\begin{theorem}\label{main-thm-auto}(cf.~Theorem \ref{thm-auto})
Let $X$ be a rationally connected smooth projective variety admitting a polarized (or int-amplified) endomorphism.
Then we have:
\begin{itemize}
\item[(1)]
$\Aut(X)/\Aut_0(X)$ is a finite group. More precisely, $\Aut(X)$ is a linear algebraic group
(with only finitely many connected components).
\item[(2)]
Every amplified endomorphism of $X$ is int-amplified.
\item[(3)]
$X$ has no automorphism of positive entropy (nor amplified automorphism).
\end{itemize}
\end{theorem}

\begin{remark}\label{rem-Intro1}
$ $
\begin{itemize}
\item[(1)]
The assumption of $X$ being rationally connected smooth in Theorems \ref{main-thm-rc} and \ref{main-thm-auto} can be weakened as in Theorems \ref{thm-q-pi-pt} and \ref{thm-auto}.
\item[(2)]
Let $X$ be a projective variety with $f\in \IAmp(X)$ and $g\in\SEnd(X)$.
Then both $f^i\circ g$ and $g\circ f^i$ are in $\IAmp(X)$ when $i \ge N$ for some $N > 0$;
see \cite[Proposition 1.4]{Meng}.
However, this $N$ may depend on $f$ and $g$.
\end{itemize}
\end{remark}

\begin{example}\label{ex:power_eq_map}
\rm{}
Let $X:=\mathbb{P}^1\times \mathbb{P}^1$.
We define endomorphisms $f, g$ on $X$ as:
$$\begin{aligned}
f([a_1:b_1],[a_2:b_2]) &=([a_2:b_2],[a_1^4:b_1^4]), \\
g([a_1:b_1],[a_2:b_2]) &=([a_2^4:b_2^4],[a_1:b_1]).
\end{aligned}$$
Denote by $h=g\circ f$.
Then $$h([a_1:b_1],[a_2:b_2])=([a_1^{16}:b_1^{16}],[a_2:b_2]).$$
Note that $f^2([a_1:b_1],[a_2:b_2])=g^2([a_1:b_1],[a_2:b_2])=([a_1^4:b_1^4],[a_2^4:b_2^4])$.
Clearly, $f$ and $g$ are then $2$-polarized, but $h$ is not int-amplified.
Note also that the set of preperiodic points of $f$ and $g$ are the same.
\end{example}

\par \vskip 1pc
{\bf The difference with early papers.}
In \cite{MZ} for polarized $f \in \SEnd(X)$ and \cite{Meng} for int-amplified $f$, it was
proved that the single $f$, replaced by a power, fixes a $K_X$-negative extremal ray.
In this paper, we prove that there are only finitely many (not necessarily $K_X$-negative) 
contractible extremal rays. This guarantees the MMP is $\SEnd(X)$-equivariant; 
and even the whole monoid $\SEnd(X)$ (all up to finite-index) is diagonalizable (and hence commutative) over $\NS_{\Q}(X)$
when $X$ is smooth rationally connected.

Even when $X$ has Picard number one,
the following question is still open when $n \ge 4$.

\begin{question}
Let $X$ be a rationally connected smooth projective variety of dimension $n \ge 1$ 
which admits a polarized endomorphism.
Is $X$ (close to be) a toric variety?
\end{question}

\par \vskip 1pc
{\bf Acknowledgement.}
The authors would like to thank the referee for very careful reading, constructive suggestions, and pointing out the necessity to add the normality assumption in Lemmas \ref{lem-omt} and \ref{lem-open}. 
The first named-author is supported by a Research Assistantship of NUS.
The second named-author is supported by an Academic Research Fund of NUS.

\section{Preliminaries}
Throughout this section, we work over an arbitrary algebraically closed field $k$.

\par \vskip 1pc
{\bf Terminology and notation.}
Let $X$ be a projective variety.
A Cartier divisor is always integral, unless otherwise indicated.

Let $n:=\dim(X)$.
We can regard $\N^1(X):=\NS(X)\otimes_{\Z}\R$ as the space of {\it numerically equivalent classes} of $\R$-Cartier divisors.
Denote by $\N_r(X)$ the space of {\it weakly numerically equivalent classes} of $r$-cycles with $\R$-coefficients (cf.~\cite[Definition 2.2]{MZ}).
Denote by $\NE(X)$ the cone of the closure of effective real $1$-cycles in $\N_1(X)$.
When $X$ is normal, we also call $\N_{n-1}(X)$ the space of {\it weakly numerically equivalent classes} of Weil $\R$-divisors.
In this case, $\N^1(X)$ can be regarded as a subspace of $\N_{n-1}(X)$ (cf.~\cite[Lemma 3.2]{Zh-tams}).
For $\K := \Q$, $\R$, or $\C$, denote by $\NS_{\K}(X):=\NS(X)\otimes_\Z \K$.

\begin{definition}\label{def-G-equiv}
Let
$$(*):X_1\dasharrow X_2\dasharrow\cdots\dasharrow X_r$$
be a finite sequence of dominant rational maps of projective varieties.
Let $f:X_1\to X_1$ be a surjective endomorphism.
We say the sequence $(*)$ is {\it $f$-equivariant} if the following diagram is commutative
$$\xymatrix{
X_1\ar@{-->}[r]\ar[d]^{f_1} &X_2\ar@{-->}[r]\ar[d]^{f_2} &\cdots\ar@{-->}[r] &X_r\ar[d]^{f_r}\\
X_1\ar@{-->}[r] &X_2\ar@{-->}[r] &\cdots\ar@{-->}[r] &X_r\\
}
$$
where $f_1=f$ and all $f_i$ are surjective endomorphisms.
Let $G$ be a subset of $\SEnd(X_1)$.
We say the sequence $(*)$ is {\it $G$-equivariant} if $(*)$ is $g$-equivariant for any $g\in G$.
\end{definition}

\begin{definition} Let $f:X\to X$ be a surjective endomorphism of a projective variety $X$.
\begin{itemize}
\item[(1)] $f$ is $q$-{\it polarized}  if $f^{\ast}L \sim qL$ for some ample Cartier divisor $L$ and integer $q>1$.
\item[(2)] $f$ is {\it amplified} if $f^*L-L=H$ for some Cartier divisor $L$ and ample
divisor $H$.
\item[(3)] $f$ is {\it int-amplified} if $f^*L-L=H$ for some ample Cartier divisors $L$ and $H$.
\item[(4)] $f$ is {\it separable} if the induced field extension $f^*:k(X)\to k(X)$ is separable where $k(X)$ is the function field of $X$.
\end{itemize}
\end{definition}

Let $f:X\to X$ be a surjective endomorphism of a projective variety $X$ of dimension $n \ge 1$.
Let $L$ be a Cartier divisor of $X$.
Then
$$(f^s)^*L-L=f^*L'-L'=\sum\limits_{i=0}^{s-1}(f^i)^*(f^*L-L)$$
where $L'=\sum\limits_{i=0}^{s-1}(f^i)^*L$.
Therefore, $f$ is amplified (resp.~int-amplified) if and only if so is $f^s$ for some (or all) $s>0$.
Suppose further $q := (\deg f)^{\frac{1}{n}}$ is rational (and hence an integer).
If $(f^s)^*L\sim q' L$ for some ample Cartier divisor $L$ and $q' >0$, then $q' = (\deg f^s)^{\frac{1}{n}} = q^s$ and $f^*L'' \sim q L''$ where $L'' =\sum\limits_{i=0}^{s-1} \,
q^{s-i}(f^i)^*L$.
Therefore, $f$ is polarized if and only if so is $f^s$ for some (or all) $s>0$.

\begin{definition}
Let $X$ be a projective variety.
\begin{itemize}
\item[(1)] $\SEnd(X)$ is the monoid of surjective endomorphisms of $X$.
\item[(2)] $\Pol(X)$ is the set of polarized endomorphisms of $X$.
\item[(3)] $\IAmp(X)$ is the set of int-amplified endomorphisms of $X$.
\end{itemize}
\end{definition}

We thank the referee to point out that the assumption of normality in the below two lemmas is necesssary and give the reference.
\begin{lemma}(cf.~\cite[Theorem, Page 220]{Dan})\label{lem-omt}
Let $f:X\to Y$ be a finite surjective morphism of two varieties with $Y$ being normal. Then $f$ is an open map.
\end{lemma}

By the above lemma, one easily gets the following result.
\begin{lemma}(cf.~\cite[Lemma 7.2]{CMZ})\label{lem-open} Let $f:X\to Y$ be a finite surjective morphism of two varieties with $Y$ being normal.
Let $S$ be a subset of $Y$.
Then $f^{-1}(\overline{S})=\overline{f^{-1}(S)}$.
\end{lemma}

Next we prepare some useful lemmas about (int-)amplified endomorphisms.

\begin{lemma}\label{lem-cyc} Let $f:X\to X$ be an int-amplified endomorphism of a normal projective variety $X$ of dimension $n$.
Suppose $f^*Z\equiv_w aZ$ (weakly numerical equivalence) for some real number $a$ and effective $r$-cycle $Z\in \N_r(X)$ with $r<n$.
Then either $Z=0$ or $a>1$.
\end{lemma}
\begin{proof}
Let $H$ be any ample Cartier divisor.
By \cite[Lemma 3.11]{Meng}, $$0=\lim\limits_{i\to +\infty} Z\cdot\frac{(f^i)^*(H^r)}{(\deg f)^i}=\lim\limits_{i\to +\infty} \frac{1}{a^i}(f^i)^*Z\cdot\frac{(f^i)^*(H^r)}{(\deg f)^i}=\lim\limits_{i\to +\infty} \frac{1}{a^i}Z\cdot H^r.$$
Suppose $Z\neq 0$.
Since $Z$ is effective, $Z\cdot H^r>0$ and $a>0$.
Therefore, $a>1$.
\end{proof}

\begin{lemma}\label{lem-amp-per}
Let $\pi:X\dasharrow Y$ be a dominant map of projective varieties.
Let $f:X\to X$ and $g:Y\to Y$ be two surjective endomorphisms such that $g\circ\pi=\pi\circ f$.
Suppose $f$ is amplified.
Then $\Per(g)$ is Zariski dense in $Y$.
\end{lemma}
\begin{proof}
Let $U$ be an open dense subset of $X$ such that $\pi|_U$ is well defined.
By \cite[Theorem 5.1]{Fak}, $\Per(f)\cap U$ is Zariski dense in $X$ and hence $\pi(\Per(f)\cap U)$ is Zariski dense in $Y$.
Note that $\pi(\Per(f)\cap U)\subseteq \Per(g)$.
So the lemma is proved.
\end{proof}

\begin{lemma}\label{lem-amp-id}
Let $\pi:X\dasharrow Y$ be a dominant map of projective varieties.
Let $f:X\to X$ be an amplified endomorphisms such that $\pi=\pi\circ f$.
Then $\dim(Y)=0$.
\end{lemma}

\begin{proof}
We may assume $X$ is over the field $k$ which is uncountable by taking the base change.
Let $U$ be an open dense subset of $X$ such that $\pi|_U$ is well-defined.
Let $W$ be the graph of $\pi$ and $p_1:W\to X$ and $p_2:W\to Y$ the two projections.
For any closed point $y\in Y$, denote by $X_y:=p_1(p_2^{-1}(y))$ and $U_y:=U\cap X_y$.
Note that $U_{y_1}\cap U_{y_2}=\emptyset$ if $y_1\neq y_2$.
By assumption, $f^{-1}(X_y)=X_y$.
Then for some $s_y>0$, $f^{-s_y}(X_y^i)=X_y^i$ for every irreducible component $X_y^i$ of $X_y$, and $f^{s_y}|_{X_y^i}$ is amplified (cf.~\cite[Lemma 2.3]{Meng}).
If $U_y\neq\emptyset$, then $\Per(f)\cap U_y=\Per(f|_{X_y})\cap U_y=\bigcup_i\Per(f^{s_y}|_{X_y^i})\cap U_y\neq \emptyset$ by \cite[Theorem 5.1]{Fak}.
Suppose $\dim(Y)>0$.
There are uncountably many $y\in Y$ such that $U_y\neq\emptyset$ and $\Per(f)\supseteq\bigcup_{y\in Y} (\Per(f)\cap U_y)$.
In particular, $\Per(f)$ is uncountable, a contradiction to \cite[Lemma 2.4]{Meng}.
\end{proof}

We don't know whether the ``amplified" property is preserved via an equivariant descending. Nevertherless, the following result is enough during the proof of Theorem \ref{main-thm-GMMP}.

\begin{lemma}\label{lem-amp-aut}
Let $\pi:X\dasharrow Y$ be a dominant map of projective varieties,
$f:X\to X$ and $g:Y\to Y$ two surjective endomorphisms such that $g\circ\pi=\pi\circ f$, and
$Z$ a closed subvariety of $Y$ such that $g(Z)=Z$.
Suppose $f$ is amplified, $\dim(Z)>0$ and
$\pi$ is well defined over an open dense subset $U\subseteq X$ such that $\pi|_U^{-1}(Z)\neq\emptyset$.
Then $g|_Z\not\in \Aut_0(Z)$.
\end{lemma}
\begin{proof}
Let $W$ be the graph of $\pi$ and $p_1:W\to X$ and $p_2:W\to Y$ the two projections.
Denote by $X':=p_1(p_2^{-1}(Z))$.
Then $f(X')\subseteq X'$.
Since $\pi|_U^{-1}(Z)\neq\emptyset$, there exists at least one irreducible component $X'_i$ of $X'$ dominating $Z$ via $\pi$.
If $X'_i$ dominates $Z$, then $f(X'_i)$ dominates $Z$.
Replacing $f$ by some positive power, we may assume $f(X'_i)=X'_i$ for some $X'_i$ dominating $Z$.
Note that $f|_{X'_i}$ is still amplified (cf.~\cite[Lemma 2.3]{Meng}).
Therefore, it suffices for us to consider the case when $Z=Y$.

Suppose the contrary that $g\in \Aut_0(Y)$.
Let $G$ be the closure of the group generated by $g$ in $\Aut_0(Y)$.
Let $\tau:Y\dasharrow Y':=Y/G$.
Then $\tau=\tau\circ g$.
By Lemma \ref{lem-amp-id}, $Y'$ is a point.
Then the orbit $Gy$ is open dense in $Y$ for some $y\in Y$.
By Lemma \ref{lem-amp-per}, we may assume $y\in \Per(g)$.
Then $Gy$ is a finite set and hence $\dim(Y)=0$, a contradiction.
\end{proof}

\section{Totally periodic subvarieties}

Throughout this section, we work over an arbitrary algebraically closed field $k$.

Let $f:X\to X$ be a surjective endomorphism of a normal projective variety $X$ and $S$ a subset of $X$.
Here, a subset $S$ of $X$ is always a set of closed points.
We say $S$ is $f$-{\it invariant} (resp.~$f$-{\it periodic})
if $f(S)=S$ (resp.~$f^{r}(S) = S$ for some $r \ge 1$).
We say $S$ is $f^{-1}$-{\it invariant} (resp.~$f^{-1}$-{\it periodic})
if $f^{-1}(S)=S$ (resp.~$f^{-r}(S) = S$ for some $r \ge 1$).

\begin{lemma}\label{lem-S-closed}
Let $f:X\to X$ be a surjective endomorphism of a projective variety $X$ and $Z$ a Zariski closed subset of $X$.
Then $Z$ is $f^{-1}$-periodic if and only if so is any irreducible component of $Z$.
\end{lemma}
\begin{proof}
Let $Z=\bigcup_{1\le i\le n} Z_i$ be the irreducible decomposition of $Z$.
If $f^{-s_i}(Z_i)=Z_i$ for some $s_i>0$,
then $f^{-s}(Z)=Z$ with $s=\prod_{i=1}^n s_i$.

Suppose $f^{-s}(Z)=Z$ for some $s>0$.
Then $f^{-s}$ induces a permutation on the finite set $\{Z_i\}_{i=1}^n$.
Therefore, $f^{-sn!}(Z_i)=Z_i$ for each $i$.
\end{proof}

\begin{definition}\label{def-sigma_f}
Let $f:X\to Y$ be a separable finite surjective morphism of two normal projective varieties.
Denote by $R_f$ the {\it ramification divisor} of $f$ (cf.~\cite[Lemma 4.4]{Ok}), and $\Sigma_f$ the union of the prime divisors in $R_f$.
\end{definition}

\begin{lemma}\label{singrf} Let $f:X\to X$ be an int-amplified separable endomorphism of a normal projective variety $X$. Let $Z$ be an $f^{-1}$-periodic irreducible closed subvariety such that $Z\subsetneq X$. Then $f^{-i}(Z)\subseteq \Sing(X)\cup \Sigma_f$ for some $i \ge 0$.
\end{lemma}
\begin{proof} We may assume $\dim(X)>0$. Suppose $f^{-m}(Z)=Z$ for some $m>0$. Let $Z_i=f^{-i}(Z)$, which is irreducible. If $Z_i\not\subset \Sing(X)\cup \Sigma_f$ for each $i$, then $Z_i=f^\ast Z_{i-1}$ by the purity of branch loci and hence $(f^m)^\ast Z=Z$.
By Lemma \ref{lem-cyc}, $Z= 0$, a contradiction.
\end{proof}

Following the proof of \cite[Lemma 6.1]{MZ}, \cite[Lemma 6.2]{CMZ} and \cite[Lemma 8.1]{Meng}, we have the key lemma below.
As shown in \cite[Remark 6.3]{CMZ}, the following condition (2) is necessary.

\begin{lemma}\label{lem-MA-int-p}
Let $f:X\to X$  be an int-amplified separable endomorphism of a projective variety $X$ over the field $k$ of characteristic $p\ge 0$.
Assume $A\subseteq X$ is an irreducible closed subvariety with $f^{-i}f^i(A) = A$ for
all $i\ge 0$.
Assume further either one of the following conditions.
\begin{itemize}
\item[(1)] $A$ is a prime divisor of $X$.
\item[(2)] $p > 0$ and co-prime with $\deg f$.
\item[(3)] $p=0$.
\end{itemize}
Then $M(A) := \{f^i(A)\,|\, i \in \Z\}$ is a finite set.
\end{lemma}
\begin{proof}
The proof follows from the proof of \cite[Lemma 6.1]{MZ}, \cite[Lemma 6.2]{CMZ} and \cite[Lemma 8.1]{Meng}.
The only thing we need to check is that if condition (2) holds and $Z$ is an $f^{-1}$-invariant closed subvariety of $X$, then $p$ and $\deg f|_Z$ are co-prime.
Let $d_1=\deg f$ and $d_2=\deg f|_Z$.
Then $f_*Z=d_2 Z$.
Suppose $f^*Z=aZ$ for some integer $a>0$.
By the projection formula, $ad_2=d_1$.
Then $p$ and $d_2$ are co-prime.
\end{proof}

\begin{lemma}\label{lem-finite-orbit}
Let $f:X\to X$  be an int-amplified separable endomorphism of a projective variety $X$ over the field $k$ of characteristic $p\ge 0$.
Assume $A\subseteq X$ is a Zariski closed subset with $f^{-i}f^i(A) = A$ for
all $i\ge 0$.
Assume further either one of the following conditions.
\begin{itemize}
\item[(1)] $A$ is a reduced divisor of $X$.
\item[(2)] $p > 0$ and co-prime with $\deg f$.
\item[(3)] $p=0$.
\end{itemize}
Then each irreducible component $A_k$ of $A$ is $f^{-1}$-periodic.
In particular, $A$ is $f^{-1}$-periodic.
\end{lemma}
\begin{proof}Choose $i_0 \ge 0$ such that $A':= f^{i_0}(A), f(A'), f^2(A'), \cdots$
all have the same number of irreducible components.
Then $f^{-i}f^i(A'_k) =A'_k$ for every irreducible component $A'_k$ of $A'$.
Now the lemma follows from Lemmas \ref{lem-MA-int-p} and \ref{lem-S-closed}.
\end{proof}

We use Proposition \ref{prop-finiteclosed} below in proving the results in the introduction.
As kindly informed by Professors Dinh and Sibony, when $k = \C$, this kind of result (with a complete proof) first appeared
in \cite[Section 3.4]{DS03}; \cite[Theorem 3.2]{Dinh09} is a more general form including Proposition \ref{prop-finiteclosed} below, requiring a weaker condition and dealing with also dominant meromorphic self-maps of K\"ahler manifolds;
see comments in \cite[page 615]{DS10} for the history of these results.

Here we offer a slightly more algebraic proof and
it works also over any algebraically closed field $k$ with $p = \ch k$ co-prime to $\deg f$
(so that the usual ramification divisor formula is applicable to $f$ and its restrictions to subvarieties stable under the action of the powers of $f$).
The assumption that $p = \ch k$ and $\deg f$ are co-prime is necessary;
see Example \ref{exa-inf-per}.

\begin{proposition}\label{prop-finiteclosed}(see \cite[Section 3.4]{DS03}, \cite[Theorem 3.2]{Dinh09} and comments in \cite[page 15]{DS10}; see also \cite{BD})
Let $f:X\to X$  be an int-amplified endomorphism of a projective variety $X$ over the field $k$ of characteristic $p\ge 0$.
Suppose either $p=0$, or $p$ and $\deg f$ are co-prime.
Then there are only finitely many $f^{-1}$-periodic Zariski closed subsets.
\end{proposition}

\begin{proof}
By taking normalization, we may assume $X$ is normal.
If $S$ is an $f^{-1}$-periodic Zariski closed subsets, then each irreducible component of $S$ is $f^{-1}$-periodic by Lemma \ref{lem-S-closed}.
So it suffices to show that $X$ has only finitely many $f^{-1}$-periodic irreducible closed subvarieties.

We prove by induction on $\dim(X)$. It is trivial if $\dim(X)=0$.
Suppose the contrary that
there are infinitely many $f^{-1}$-periodic proper closed subvarieties of the same dimension $d$.
Then we may find an infinite sequence of $f^{-1}$-periodic proper closed subvarieties $S_i$ of the same dimension $d$ with $S_i\subseteq \Sing(X)\cup \Sigma_f$ by Lemma \ref{singrf}.
Let $Y$ be the closure of the union of $S_i$. Then $Y\subseteq \Sing(X)\cup \Sigma_f$.
By Lemma \ref{lem-open}, for any $j\ge 0$, $f^{-j}f^j(Y)=f^{-j}f^j(\overline{\bigcup S_i})=f^{-j}(\overline{f^j(\bigcup S_i)})=f^{-j}(\overline{\bigcup f^j(S_i)})=\overline{f^{-j}(\bigcup f^j(S_i))}=\overline{\bigcup f^{-j} f^j(S_i)}=\overline{\bigcup S_i}=Y$.
Let $Y_k$ be the irreducible component of $Y$. By Lemma \ref{lem-finite-orbit}, we may assume $f^{-1}(Y_k)=Y_k$ after replacing $f$ by some positive power.
Note that $f|_{Y_k}$ is int-amplified and $\dim(Y_k)<\dim(X)$. Then for each $k$, $Y_k$ contains finitely many $S_i$ by induction. This is a contradiction.
\end{proof}

\begin{example}\label{exa-inf-per}
\rm{}
Let $X:=\mathbb{P}_k^3$ with $p= \ch \, k=3$.
Let $f:X\to X$ via $$f([a:b:c:d])=[a^3+acd:b^3+bcd:c^3+c^2d:d^3-cd^2].$$
Then $f$ is $3$-polarized and separable.
Let $X_1:=\{c=0, d=0\}\cong \mathbb{P}^1$.
Then $f^{-1}(X_1)=X_1$ and $f|_{X_1}([a:b])=[a^3:b^3]$ which is a geometric Frobenius of $\mathbb{P}^1$.
Note that $f|_{X_1}$ is polarized and bijective.
When $a$ is a $(3m-1)$-th root of unity for some $m>0$ and $b$ is a $(3n-1)$-th root of unity for some $n>0$, the point $[a:b:0:0]$ is $f$-periodic and hence $f^{-1}$-periodic.
In particular, there are infinitely many $f^{-1}$-periodic closed points in $X$.
\end{example}

A Zariski-open subset of Zariski-closed subvariety of $X$ is called
a {\it subvariety} of $X$.

\begin{corollary}\label{cor-finite} Let $f:X\to X$  be an int-amplified separable endomorphism of a projective variety $X$ over the field $k$ of characteristic $p\ge 0$.
Suppose either $p=0$, or $p$ and $\deg f$ are co-prime.
Then $X$ has only finitely many
(not necessarily closed) $f^{-1}$-periodic subvarieties.
\end{corollary}
\begin{proof} By taking normalization, we may assume $X$ is normal.
If $A$ is $f^{-1}$-periodic, then so are $\overline{A}$ and $\overline{A}-A$ by Lemma \ref{lem-open}. Note that $\overline{A}-A$ is a Zariski closed subset of $X$. If $X$ has infinitely many $f^{-1}$-periodic subvarieties $S_i$, then we may assume $\overline{S_i}-S_i\neq \emptyset$ with $\overline{S_i}=\overline{S_j}$ for any $i, j$ by Proposition \ref{prop-finiteclosed}.  If $\overline{S_i}-S_i=\overline{S_j}-S_j$, then $S_i=S_j$. Hence, $X$ has infinitely many $f^{-1}$-periodic Zariski closed subsets $\overline{S_i}-S_i$, a contradiction to Proposition \ref{prop-finiteclosed}.
\end{proof}

\section{Equivariant MMP and proof of Theorem \ref{main-thm-finite-R}}

In this section, we work over an algebraically closed field $k$ of characteristic $0$.
We prove Theorems \ref{prop-finite-cont-r}
and \ref{thm-GMMP} which include Theorem
\ref{main-thm-finite-R}.

Let $X$ be a projective variety and let $C$ be a curve.
Denote by $R_C:=\mathbb{R}_{\ge 0}[C]$ the ray generated by $[C]$ in $\NE(X)$.
Denote by $\Sigma_C$ the union of curves whose classes are in $R_C$.

\begin{definition}\label{def:extrem_ray} Let $X$ be a projective variety.
Let $C$ be a curve such that $R_C$ is an extremal ray in $\NE(X)$.
We say $C$ or $R_C$ is {\it contractible} if there is a surjective morphism
$\pi : X \to Y$ to a projective variety $Y$ such that
the following hold.

\begin{itemize}
\item[(1)] $\pi_*\mathcal{O}_X=\mathcal{O}_Y$.
\item[(2)] Let $C'$ be a curve in $X$.
Then $\pi(C')$ is a point if and only if $[C'] \in R_C$.
\item[(3)] Let $D$ be a $\Q$-Cartier divisor of $X$.
Then $D\cdot C=0$ if and only if $D\equiv \pi^*D_Y$ (numerical equivalence) for some $\Q$-Cartier divisor $D_Y$ of $Y$.
\end{itemize}
\end{definition}

If $R_C$ is an extremal ray contracted by $\pi$, then $\Sigma_C$ equals $\Exc(\pi)$ which is Zariski closed in $X$; here $\Exc(\pi)$ is the exceptional locus of $\pi$ (i.e. the subset of $X$ along which $\pi$ is not an isomorphism).

When $(X, \Delta)$ is lc, every $(K_X + \Delta)$-negative extremal ray $R_C$ is contractible.

\begin{lemma}(cf.~\cite[Lemma 2.11]{Zh-comp})\label{lem-ray}
Let $X$ be a projective variety and let $R_C$ be a ray of $\NE(X)$ generated by some curve $C$.
Let $h\in \SEnd(X)$.
Then we have:
\begin{itemize}
\item[(1)]
$h_*(R_C)=R_{h(C)}$ and $h^*(R_C)=R_{C'}$ for any curve $C'$ with $h(C')=C$.
\item[(2)]
$h(\Sigma_C)=\Sigma_{h(C)}$ and $h^{-1}(\Sigma_C)=\Sigma_{C'}$ for any curve $C'$ with $h(C')=C$.
\item[(3)]
$R_C$ is extremal if and only if so is $R_{h(C)}$ for some $h\in \SEnd(X)$, if and only if so is $R_{h(C)}$ for any $h\in \SEnd(X)$.
\end{itemize}
Suppose $R_C$ is extremal.
\begin{itemize}
\item[(4)]
If $R_{h(C)}$ is contractible, then so is $R_C$.
\end{itemize}
\end{lemma}

\begin{proof}
Let $h\in \SEnd(X)$.
Note that $h_*$ and $h^*$ are invertible linear selfmaps of $\N_1(X)$ and $h_*^{\pm}(\NE(X))={h^*}^{\pm}(\NE(X))=\NE(X)$.
Note that $h_*C=(\deg h|_C) h(C)$.
So $h_*(R_C)=R_{h(C)}$.
Since $(h_*\circ h^*)|_{\N_1(X)}=(\deg h)\,\id$, 
$h_*(R_{C'})=R_{h(C')}=R_{h(C)}$ implies $h^*(R_C)=R_{C'}$ for any curve $C'$ with $h(C')=C$.
So (1) is proved.

For any curve $E$ with $[E]\in R_C$, $[h(E)]\in R_{h(C)}$.
Then $h(\Sigma_C)=h(\bigcup_{[E]\in R_C} E)=\bigcup_{[E]\in R_C} h(E)\subseteq \Sigma_{h(C)}$.
For any curve $F$ with $[F]\in R_{h(C)}$, there is some curve $F_1$ such that $h(F_1)=F$.
Note that $R_{F_1}=h^*(R_F)=h^*(R_{h(C)})=R_C$ by (1).
So $[F_1]\in R_C$ and hence $h(\Sigma_C)=\Sigma_{h(C)}$.
Similarly, $h^{-1}(\Sigma_C)=\Sigma_{C'}$.
So (2) is proved.

By (1), $R_{h(C)}=h_*(R_C)$. Note that the set of extremal rays are stable under the actions $h_*$ and $h^*$. So (3) is straightforward.

For (4), suppose $R_{h(C)}$ is extremal and contractible by $\pi:X\to Y$.
Taking the Stein factorization of $\pi\circ h$, we have $\pi':X\to Y'$ and $\tau:Y'\to Y$ such that $\pi'_*\mathcal{O}_X=\mathcal{O}_Y$ and $\tau$ is a finite surjective morphism.
We claim that $\pi'$ is the contraction of $R_C$.
For any curve $C'$ on $X$, $\pi'(C')$ is a point if and only if $\pi(h(C'))$ is a point;
if and only if $[h(C')]\in R_{h(C)}$; if and only if $[C']\in R_C$ by (1).
Let $D'$ be a $\Q$-Cartier divisor of $X$ such that $D'\cdot C=0$.
Since $h^*|_{\NS_{\Q}(X)}$ is invertible, $D'\equiv h^*D$ for some $\Q$-Cartier divisor $D$.
By the projection formula, $D\cdot h(C)=0$.
Since $\pi$ is the contraction of $h(C)$, $D\equiv \pi^*D_Y$ for some $\Q$-Cartier divisor $D_Y$ of $Y$.
Then $D'\equiv h^*(\pi^*D_Y)=\pi'^*(\tau^*D_Y)$.
So the claim is proved.
\end{proof}

\begin{lemma}\label{lem-sigma-per}
Let $f: X \to X$ be an int-amplified endomorphism of a projective variety.
Let $h\in \SEnd(X)$.
Let $R_C$ be a contractible extremal ray of $\NE(X)$ and $F$ an irreduicble component of $\Sigma_C$.
Then we have:
\begin{itemize}
\item[(1)] $h^i(\Sigma_C)$ and $h^i(F)$ are $f^{-1}$-periodic for any $i\in \Z$.
\item[(2)] $\Sigma_C$ and $F$ are $h^{-1}$-periodic.
\end{itemize}
\end{lemma}

\begin{proof}
Let $h\in \SEnd(X)$, $C'=h(C)$ and $C=h(\widetilde{C})$ for some curve $\widetilde{C}$.
Since $R_C$ is contractible, $\Sigma_C$ is Zariski closed in $X$.
By Lemma \ref{lem-ray}, $h(\Sigma_C)=\Sigma_{C'}$ and $h^{-1}(\Sigma_C)=\Sigma_{\widetilde{C}}$ are Zariski closed in $X$; and for any $j\ge 0$, $f^{-j}f^j(\Sigma_{C'})=\Sigma_{C'}$ and $f^{-j}f^j(\Sigma_{\widetilde{C}})=\Sigma_{\widetilde{C}}$ .
By Lemma \ref{lem-finite-orbit}, $h(\Sigma_C)$ and $h^{-1}(\Sigma_C)$ are both $f^{-1}$-periodic.
By Lemma \ref{lem-S-closed}, $h(F)$ and $h^{-1}(F)$ are then $f^{-1}$-periodic.
So (1) is proved.

Note that there are only finitely many $f^{-1}$-periodic Zariski closed subsets in $X$ by Proposition \ref{prop-finiteclosed}.
We have $h^m(F)=h^n(F)$ for some $m<n<0$.
So $h^{m-n}(F)=F$ and $\Sigma_C$ is $h^{-1}$-periodic by Lemma \ref{lem-S-closed}.
So (2) is proved.
\end{proof}

Following \cite[Lemma 6.2]{MZ}, we may further have the following stronger result.

\begin{lemma}\label{lem-finite-cont-r}
Let $f: X \to X$ be an int-amplified endomorphism of a projective variety $X$.
Let $E\subseteq X$ be a Zariski closed subset and let $\mathcal{R}_E$ be the set of all contractible extremal rays $R_C$ with $\Sigma_C=E$.
Then we have
\begin{itemize}
\item[(1)]
$\mathcal{R}_E$ is a finite set with $\sharp \mathcal{R}_E\le \dim(E)$.
\item[(2)]
Let $F$ be an irreducible component of $E$.
Then $$\mathcal{R}_E^F:=\{R_{h(C)}\,|\, R_C\in \mathcal{R}_E, h\in \SEnd(X), h^{-1}(F)=F\}$$ is a finite set with $\sharp \mathcal{R}_E^F\le \dim(F)$.
\end{itemize}
\end{lemma}

\begin{proof}
We assume that $\mathcal{R}_E$ is non-empty.
Let $R_C\in \mathcal{R}_E$ (we may assume $C\subseteq F$).
We have a
contraction $\pi_C:X\to Y_C$ and a linear exact sequence
$$0\to \NS_{\C}(Y_C)\xrightarrow{\pi_C^{\ast}}\NS_{\C}(X)\xrightarrow{\cdot C} \mathbb{C}\to 0.$$
So $\pi_C^{\ast}\NS_{\C}(Y_C)$ is a subspace in $\NS_{\C}(X)$ of codimension $1$.
Let $F$ be an irreducible component of $E$.
Let $j:F\hookrightarrow X$ be the inclusion map.
For any $\C$-Cartier divisor $D$ of $X$,
denote by $D|_F:=j^*D\in \NS_{\mathbb{C}}(F)$ the pullback.
Let $$\NS_{\mathbb{C}}(X)|_F:=j^*(\NS_{\mathbb{C}}(X))$$ which is a subspace of $\NS_{\mathbb{C}}(F)$.
Denote by $$L_{C}:=\{D|_F\,:\, D\in \NS_{\mathbb{C}}(X), D\cdot C=0\}.$$
Then $L_{C}=j^*\pi_C^*(\NS_{\mathbb{C}}(Y_C))$ is a subspace in $\NS_{\mathbb{C}}(X)|_F$ of codimension at most $1$.
Note that for an ample divisor $H$ in $X$, $H|_F\cdot C=H\cdot C\neq 0$.
Therefore, $H|_F\not\in L_{C}$ and hence $L_{C}$ has codimension $1$ in $\NS_{\mathbb{C}}(X)|_F$.
Denote by $$S:=\{D|_F\in \NS_{\mathbb{C}}(X)|_F\,:\,(D|_F)^{\dim(F)}=0\}.$$

We claim that $S$ is a hypersurface (an algebraic set defined by a non-zero polynomial) in
the complex affine space $\NS_{\mathbb{C}}(X)|_F$ and each $L_{C}$ is an irreducible component of $S$
in the sense of Zariski topology.
Indeed, let $\{e_1,\cdots,e_k\}$ be a fixed basis of $\NS_{\mathbb{C}}(X)|_F$.
Then $$S=\{(x_1,\cdots, x_k)\,:\, (\sum\limits_{i=1}^k x_ie_i)^{\dim(F)}=0\}$$
 is determined by a homogeneous polynomial of degree $\dim(F)$ and the coefficient of the term $\prod_i x_i^{\ell_i}$ is the intersection number $e_1^{\ell_1}\cdots e_k^{\ell_k}$.
Note that for an ample divisor $H$ in $X$, $H|_F\in \NS_{\mathbb{C}}(X)|_F$ and $(H|_F)^{\dim(F)}=H^{\dim(F)}\cdot F>0$.
So $e_1^{\ell_1}\cdots e_k^{\ell_k}\neq 0$ for some $\ell_i$.
In particular, $S$ is determined by a non-zero polynomial.
Since $\dim(\pi_C(F))<\dim(F)$, ${\pi_C}_*F=0$.
For any $P\in \NS_{\mathbb{C}}(Y)$, we have
$$(\pi_C^*P|_F)^{\dim(F)}=(\pi_C^*P)^{\dim(F)}\cdot F=P^{\dim(F)}\cdot {\pi_C}_*F=0$$ by the projection formula.
So $\pi_C^*P|_F\in S$. Hence $L_{C} \subseteq S$.
Since $L_{C}$ and $S$ have the same dimension, each $L_{C}$ is an irreducible component of $S$.
The claim is proved.

Let $h\in \SEnd(X)$ such that $h^{-1}(F)=F$.
The pullback $h^{\ast}$ induces an automorphism of $\NS_{\mathbb{C}}(X)|_F$.
Note that $h^{\ast}F= aF$ (as cycles) for some $a>0$, and $(h^{\ast}D)^{\dim(F)}\cdot F=\frac{\deg h}{a}D^{\dim(F)}\cdot F$.
Hence, $D\in S$ if and only if $h^{\ast}D\in S$. This implies that $S$ is $h^{\ast}$-invariant.
By the projection formula, $L_{h(C)}=(h^*)^{-1}(L_C)$ is also an irreducible component of $S$.
Note that $S$ has at most $\dim(F)$ irreducible components.
So (2) follows from the claim below. Clearly, (1) follows from (2) directly.

Let $g, g'\in \SEnd(X)$ such that $g^{-1}(F)=g'^{-1}(F)=F$.
Let $R_C\in \mathcal{R}_E$ and let $C'\subseteq F$ be another (not necessarily contractible or extremal) curve.
We claim that $R_{g(C)}=R_{g'(C')}$ if and only if $L_{g(C)}=L_{g'(C')}$.
Suppose $L_{g(C)}=L_{g'(C')}$.
Let $C_1$ be some curve such that $g(C_1)=g'(C')$.
By the projection formula, $L_{g(C)}=(g^*)^{-1}(L_C)$ and $L_{g'(C')}=(g^*)^{-1}(L_{C_1})$.
Then $L_{C}=L_{C_1}$.
Let $H$ be an ample Cartier divisor of $Y$.
Then $\pi_C^*H\cdot C_1=0$ implies that $\pi_C(C_1)$ is a point and hence $R_{C}=R_{C_1}$.
Therefore, $R_{g(C)}=R_{g(C_1)}=R_{g'(C')}$ by Lemma \ref{lem-ray}.
Another direction is trivial.
So the claim is proved.
\end{proof}

\begin{theorem}\label{prop-finite-cont-r}
Let $f: X \to X$ be an int-amplified endomorphism of a projective variety $X$.
Let $\mathcal{R}_{contr}$ be the set of all contractible extremal rays $R_C$.
Then we have:
\begin{itemize}
\item[(1)] $\mathcal{R}_{contr}$ is a finite set.
\item[(2)] The set $$\widetilde{\mathcal{R}}_{contr}:=\{(h_*)^i (R_{C})\,|\, R_C\in \mathcal{R}_{contr}, h\in \SEnd(X), i\in \mathbb{Z}\}$$ is  finite.
\item[(3)] There is a finite-index submonoid $H$ of
$\SEnd(X)$ such that $h_*(R)=h^*(R)=R$ for any $R\in \widetilde{\mathcal{R}}_{contr}$ and $h\in H$.
\end{itemize}
\end{theorem}

\begin{proof}
We use the notation in Lemma \ref{lem-finite-cont-r}.
Let $P_f$ be the set of $f^{-1}$-periodic Zariski closed subsets, which is finite by Proposition \ref{prop-finiteclosed}.
For any $R_C\in \mathcal{R}_{contr}$, $\Sigma_C\in P_f$ by Lemma \ref{lem-sigma-per}.
Then $\mathcal{R}_{contr}=\bigcup_{E\in P_f} \mathcal{R}_E$ is finite by
Lemma \ref{lem-finite-cont-r}.
So (1) is proved.

Let $\widetilde{\mathcal{R}}_{contr}^0:=\{h_*(R_{C})\,|\, R_C\in \mathcal{R}_{contr}, h\in \SEnd(X)\}$.
We first claim that $\widetilde{\mathcal{R}}_{contr}^0$ is finite.
Suppose the contrary that $\widetilde{\mathcal{R}}_{contr}^0$ is infinite.
Since $\mathcal{R}_{contr}$ is finite by (1), there exist some $R_C\in \mathcal{R}_{contr}$ and infinitely many $h_j\in \SEnd(X)$ with $j>0$ such that the set $\{{h_j}_*(R_{C})\}_{j=1}^{\infty}$ is infinite.
Let $F$ be an irreducible component of $\Sigma_C$.
By Lemma \ref{lem-sigma-per}, $h_j^{-1}(F)\in P_f$ and  $h_1^{-s}(F)=F$ for some $s>0$.
Note that $P_f$ is finite.
So we may assume $h_j^{-1}(F)=h_1^{-1}(F)$ for any $j>0$.
Let $\widetilde{h}_j:=h_j\circ h_1^{s-1}$.
Then $\widetilde{h}_j^{-1}(F)=h_1^{-s}(F)=F$.
For any $j_1, j_2>0$,
$({\widetilde{h}}_{j_1})_*(R_C)=({\widetilde{h}}_{j_2})_*(R_C)$
implies $(h_{j_1})_*(R_C)=(h_{j_2})_*(R_C)$ by Lemma \ref{lem-ray}.
In particular, the set $\{(\widetilde{h}_j)_*(R_C)\}_{j=1}^\infty$ is infinite.
However, this contradicts Lemma \ref{lem-finite-cont-r}.

Since $\widetilde{\mathcal{R}}_{contr}^0$ is finite, for any $h\in \SEnd(X)$ and $R_C\in \mathcal{R}$, $(h^m)_*(R_C)=(h^n)_*(R_C)$ for some $0<m<n$.
By Lemma \ref{lem-ray}, for any $i>0$, $(h_*)^{-i}(R_C)=(h_*)^{k(n-m)-i}(R_C)=(h^{k(n-m)-i})_*(R_C)\in \widetilde{\mathcal{R}}_{contr}^0$ for $k\gg 1$.
Then $\widetilde{\mathcal{R}}=\widetilde{\mathcal{R}}_{contr}^0$ is finite.
So (2) is proved.

Note that the monoid action of $\SEnd(X)$ on $\widetilde{\mathcal{R}}_{contr}$ via $(h',h_* (R_{C}))\to (h'\circ h)_*(R_C)$ is well defined.
So (3) is proved.
\end{proof}

\begin{theorem}\label{thm-finite-R}
Let $(X, \Delta)$ be an lc pair.
Let $f:X\to X$  be an int-amplified endomorphism.
Then we have:
\begin{itemize}
\item[(1)] The set $\mathcal{R}_{neg}$ of $(K_X + \Delta)$-negative extremal rays in $\NE(X)$ is finite.
\item[(2)] The set $$\widetilde{\mathcal{R}}_{neg}:=\{(h_*)^i (R)\,|\, R\in \mathcal{R}_{neg}, h\in \SEnd(X), i\in \mathbb{Z}\}$$ is finite.
\item[(3)] There is a finite-index submonoid $H$ of
$\SEnd(X)$ such that $h_*(R)=h^*(R)=R$ for any $R\in \widetilde{\mathcal{R}}_{neg}$ and $h\in H$.
\end{itemize}
\end{theorem}

\begin{proof}
Let $R\in \mathcal{R}_{neg}$.
Since $(X,\Delta)$ is lc, $R=R_C$ for some curve $C$ and $R$ is contractible by the Cone theorem in \cite[Theorem 1.1]{Fu11}.
Then we are done by Theorem \ref{prop-finite-cont-r}.
\end{proof}

\begin{theorem}\label{thm-GMMP}
Let $f:X\to X$  be an int-amplified endomorphism of a $\Q$-factorial normal projective variety $X$.
Then any finite sequence of MMP starting from $X$ is $G$-equivariant for some finite-index submonoid $G$ of $\SEnd(X)$.
\end{theorem}
\begin{proof}
By \cite[Theorem 1.6]{Meng}
(see also \cite[Corollary 1.3]{BH}),
$X$ is lc.
Let $X:=X_1\dasharrow\cdots \dasharrow X_{s}$ be a sequence of MMP.
By \cite[Theorem 8.2]{Meng}, replacing $f$ by a positive power, we may assume the above sequence is $f$-equivariant and $f_i:=f|_{X_i}$ is int-amplified.

We show the theorem by induction on $s$.
Suppose $X:=X_1\dasharrow\cdots \dasharrow X_{s-1}$ is $G$-equivariant.
By Theorem \ref{thm-finite-R}, replacing $G$ by its finite-index submonoid, we may assume $h^*(R)=R$ for any $h\in G|_{X_{s-1}}$ and any $K_{X_{s-1}}$-negative extremal ray $R$.
If $\pi_{s-1}:X_{s-1}\dasharrow X_s$ is a divisorial contraction or a Fano contraction,
then $\pi_{s-1}$ is $G|_{X_{s-1}}$-equivariant.
If $\pi_{s-1}$ is a flip, then $\pi_{s-1}$ is $G|_{X_{s-1}}$-equivariant by further applying \cite[Lemma 3.6]{Zh-comp} (cf.~\cite[Lemma 6.6]{MZ}).
\end{proof}

\section{Proof of Theorem \ref{main-thm-GMMP} and Corollary \ref{main-cor-comp}}

Throughout this section, we work over characteristic $0$. First, we prepare the following lemmas which are frequently used in the proof of our main theorems.
\begin{lemma}\label{lem-alg-eig}
Let $f:X\to X$ be a surjective endomorphism of a projective variety.
Then all the eigenvalues of $f^*|_{\NS_{\C}(X)}$ are algebraic integers.
\end{lemma}
\begin{proof}
The action $f^*|_{\NS_{\C}(X)}$ is induced by $f^*|_{\NS(X)}$.
Note that $\NS(X)$ is a $\mathbb{Z}$-module of finite rank.
The lemma follows.
\end{proof}

\begin{lemma}\label{lem-rel-int}
Let $\pi:X\to Y$ be a surjective morphism of two projective varieties such that $\pi$ is not a finite morphism and $\pi^*{\NS_{\Q}(Y)}$ is a codimension-$1$ subspace of $\NS_{\Q}(X)$.
Let $f:X\to X$ and $g:Y\to Y$ be surjective endomorphisms such that $\pi\circ f=g\circ \pi$.
Then $f^*|_{\NS_{\Q}(X)/\NS_{\Q}(Y)}=q\,\id$ for some integer $q>0$.
\end{lemma}
\begin{proof}
Note that $\NS_{\Q}(X)/\NS_{\Q}(Y)$ is $1$-dimensional.
Then $f^*|_{\NS_{\Q}(X)/\NS_{\Q}(Y)}=q\,\id$ for some $q\in \Q$.
By Lemma \ref{lem-alg-eig}, $q$ is then an integer.
Let $H$ be an ample Cartier divisor on $X$.
Then $f^*H-qH\in \pi^*\NS_{\Q}(Y)$.
Suppose $q\le 0$.
Then $f^*H-qH$ is ample on $X$.
Since $\pi$ is not finite, there is no ample class in $\pi^*\NS_{\Q}(Y)$.
So we get a contradiction.
\end{proof}

\begin{lemma}\label{lem-amp-iamp}
Let $(X, \Delta)$ be a $\Q$-factorial lc pair.
Let $\pi:X\dasharrow Y$ be either a divisorial contraction, a flip, or a Fano contraction of a $K_X+\Delta$-negative extremal ray.
Let $f:X\to X$ and $g:Y\to Y$ be surjective endomorphisms such that $g\circ \pi=\pi\circ f$.
Suppose there are a dominant map $\tau:W\dasharrow X$ and an amplified endomorphism $h:W\to W$ such that $f\circ \tau=\tau\circ h$.
Suppose further $g$ is int-amplified.
Then $f$ is int-amplified.
\end{lemma}
\begin{proof}
During the proof, we may always replacing $f, g$ and $h$ by suitable positive powers.
If $\pi$ is birational, then $f$ is int-amplified by \cite[Lemma 3.6]{Meng}.
Suppose $\pi$ is a Fano contraction and $f$ is not int-amplified.
Then $f^*D\equiv D$ for some $D\in \NS_{\Q}(X)\backslash\pi^*\NS_{\Q}(Y)$ by \cite[Proposition 3.3]{Meng} and Lemma \ref{lem-rel-int}.
We may assume $D$ is $\pi$-ample.
Since $g$ is int-amplified, $\Per(g)$ is Zariski dense in $Y$ by \cite[Theorem 5.1]{Fak}.
Let $y\in \Per(g)$ be general and we may assume $g(y)=y$.
Then $F:=\pi^{-1}(y)$ is irreducible.
Also $f(F)=F$.
Suppose $\tau$ is well defined over an open dense subset $U\subseteq W$.
Since $F$ is over general point, $\tau|_U^{-1}(F)\neq \emptyset$.
Note that $\dim(F)>0$, $D|_F$ is ample and $(f|_F)^* (D|_F)\equiv D|_F$.
Then we may assume $f|_F\in \Aut_0(F)$ (cf.~\cite[Theorem 1.2]{MZ}, \cite[Proposition 2.2]{Li}, \cite[Theorem 4.8]{Fuj}).
However, this contradicts Lemma \ref{lem-amp-aut}.
\end{proof}

We recall \cite[Lemma 9.2]{Meng} about the diagonalizable criterion for the pullback action.
\begin{lemma}\label{lem-g-diag} Let $(X, \Delta)$ be a $\Q$-factorial lc pair.
Let $\pi:X\dasharrow Y$ be either a divisorial contraction, a flip, or a Fano contraction of a $K_X+\Delta$-negative extremal ray.
Let $f:X\to X$ and $g:Y\to Y$ be surjective endomorphisms such that $g\circ \pi=\pi\circ f$.
Suppose $g^*|_{\NS_{\C}(Y)}$ is diagonalizable. Then so is $f^*|_{\NS_{\C}(X)}$.
\end{lemma}

\begin{proof}
If $\pi$ is a flip, then $\NS_{\C}(X)=\pi^*\NS_{\C}(Y)$ and hence $f^*|_{\NS_{\C}(X)}$ is diagonalizable.
If $\pi$ is a divisorial contraction with $E$ being the $\pi$-exceptional prime divisor, then $f^*E=\lambda E$ for some integer $\lambda\ge 1$. Note that $-E$ is $\pi$-ample by \cite[Lemma 2.62]{KM}. Its class $[E]\in \NS_{\C}(X)\backslash \pi^*\NS_{\C}(Y)$.
Note that $\pi^*\NS_{\C}(Y)$ is a codimension-$1$ subspace of $\NS_{\C}(X)$.
Then $f^*|_{\NS_{\C}(X)}$ is diagonalizable.
If $\pi$ is a Fano contraction, then $f^*|_{\NS_{\C}(X)}$ is diagonalizable by \cite[Lemma 9.2]{Meng}.
\end{proof}

\begin{lemma}\label{lem-xd}
Let $(X,\Delta)$ be an lc pair and let $\pi:X\to Y$ be a Fano contraction of a $(K_X+\Delta)$-negative extremal ray.
Let $m:=\dim(X)$, $n:=\dim(Y)$ and $d:=\dim(X)-\dim(Y)$.
Let $D\in \NS_{\C}(X)$ and $H_1, \cdots, H_n\in \NS_{\C}(Y)$ such that $D^d\cdot \pi^*H_1\cdots \pi^*H_n=0$.
Then either $D\in \pi^*\NS_{\C}(Y)$ or $H_1\cdots H_n=0$.
\end{lemma}
\begin{proof}
Let $C$ be some curve contracted by $\pi$.
Suppose $D\not\in \pi^*\NS_{\C}(Y)$.
This is equivalent to saying $D\cdot C\neq 0$.
Let $A$ be a very ample Cartier divisor of $X$.
Then $(D-aA)\cdot C=0$ for some $a\neq 0$ and $E:=D-aA\in \pi^*\NS_{\C}(Y)$.
Let $Z:=A_1\cap \cdots \cap A_d$ where $A_1,\cdots, A_d$ are general members in the linear system $|A|$.
Note that $Z$ is pure $n$-dimensional and every irreducible component of $Z$ dominates $Y$.
In particular, $\pi_*(A^d)=bY$ for some $b>0$.
By the projection formula, we have
$0=D^d\cdot \pi^*H_1\cdots \pi^*H_n=(E+aA)^d\cdot \pi^*H_1\cdots \pi^*H_n=(a^d) A^d\cdot \pi^*H_1\cdots \pi^*H_n=(a^db) H_1\cdots H_n$.
Therefore, $H_1\cdots H_n=0$.
\end{proof}

Next, we provide a submonoid version of Lemma \ref{lem-g-diag}.

\begin{lemma}\label{lem-G-diag}
Let $X$ be a $\Q$-factorial lc projective variety
and let $(X, \Delta)$ be lc.
Let $\pi:X\dasharrow Y$ be either a divisorial contraction, a flip, or a Fano contraction of a $K_X+\Delta$-negative extremal ray.
Let $G$ be a subset of $\SEnd(X)$ such that $\pi$ is $G$-equivariant.
Suppose $(G|_Y)^*|_{\NS_{\C}(Y)}$ is diagonalizable.
Then so is $G^*|_{\NS_{\C}(X)}$.
\end{lemma}

\begin{proof}
If $\pi$ is a flip, the lemma is trivial.
If $\pi$ is a divisorial contraction with $E$ being the $\pi$-exceptional prime divisor, then $[E]\in \NS_{\C}(X)\backslash \pi^*\NS_{\C}(Y)$ is a common eigenvector of $h^*|_{\NS_{\C}(X)}$ for any $h\in G$, and the lemma also holds.
Next we assume $\pi$ is a Fano contraction and regard $\NS_{\C}(Y)$ as a $1$-codimensional subspace of $\NS_{\C}(X)$.
Note that for any $h\in G$, $h^*|_{\NS_{\C}(X)}$ is diagonalizable by Lemma \ref{lem-g-diag}.

Let $f,g\in G$.
Suppose $f^*x_1=ax_1$ for some $x_1\in \NS_{\C}(X)\backslash\NS_{\C}(Y)$ and $a\neq 0$.
Let $x_2,\cdots, x_k$ be a basis of $\NS_{\C}(Y)$ such that $x_2,\cdots, x_k$ are eigenvectors of $h^*|_{\NS_{\C}(Y)}$ for any $h\in G$.
Suppose $f^*x_i=a_i x_i$ with $a_i\neq 0$.
We may assume that $a_i=a$ if and only if $i\le r$ for some $r\ge 1$.
Let $g^*x_1=bx_1+y$ for some $b\neq 0$ and $y\in \NS_{\C}(Y)$.
Write $y=\sum_{i=2}^k s_i x_i$ where $s_i\in \mathbb{C}$.
Since $g^*|_{\NS_{\C}(X)}$ is diagonalizable, $s_i\neq 0$ implies $g^*x_i\neq bx_i$.
Then for each $i\le r$ such that $s_i\neq 0$, we may replace $x_1$ by $x_1+t_ix_i$ for some suitable $t_i$, such that, finally $f^*x_1=ax_1$ and $g^*x_1=bx_1+\sum_{i=r+1}^k s_ix_i$.

Next we claim $y=0$.
Set $m:=\dim(X)$, $n:=\dim(Y)$ and $d:=m-n$.
Suppose $y\neq 0$.
Then $y\cdot C\neq 0$ on $Y$ for some $C=x_2^{\ell_2}\cdots x_k^{\ell_k}$ with $\sum_{i=2}^k \ell_i=n-1$.
So $x_j\cdot C\neq 0$ on $Y$ for some $j>r$ and hence $x_1^d\cdot y\cdot C\neq 0$ and $x_1^d\cdot x_j\cdot C\neq 0$ on $X$ by Lemma \ref{lem-xd}.
Let $f^*C=eC$ and $g^*C=e'C$ for some non-zero $e$ and $e'$.
By the projection formula, $$(\deg g)x_1^{d+1}\cdot C=(g^*x_1)^{d+1}\cdot g^*C=(b^{d+1}e')x_1^{d+1}\cdot C+((d+1)b^de')x_1^d\cdot y\cdot C.$$
Since $x_1^d\cdot y\cdot C\neq 0$, we have $x_1^{d+1}\cdot C\neq 0$.
On the other hand, by the projection formula, $$(\deg f) x_1^d\cdot x_j\cdot C=(f^*x_1)^d\cdot f^*x_j\cdot f^*C=(a^{d}a_j e)x_1^d\cdot x_j\cdot C.$$
Since $x_1^d\cdot x_j\cdot C\neq 0$, we have $\deg f=a^{d}a_j e$.
By the projection formula again, $$(\deg f) x_1^{d+1}\cdot C=(f^*x_1)^{d+1}\cdot f^*C=(a^{d+1}e)x_1^{d+1}\cdot C.$$
Since $a_j\neq a$, $\deg f=a^{d}a_j e\neq a^{d+1}e$.
Hence $x_1^{d+1}\cdot C=0$, a contradiction.
So $y=0$ as claimed.

Now $y=0$ implies $f^*|_{\NS_{\C}(X)}\circ g^*|_{\NS_{\mathbb{C}}(X)}= g^*|_{\NS_{\C}(X)}\circ f^*|_{\NS_{\C}(X)}$.
So $G^*|_{\NS_{\mathbb{C}}(X)}$ is a commutative set.
Since $G^*|_{\NS_{\mathbb{C}}(X)}$ consists of diagonalizable elements,
$G^*|_{\NS_{\mathbb{C}}(X)}$ is diagonalizable by \cite[Section 15.4]{Hu}.
\end{proof}

\begin{proof}[Proof of Theorem \ref{main-thm-GMMP}]
By \cite[Theorem 1.10]{Meng}, we have an $f$-equivariant relative MMP
$$X=X_0 \dashrightarrow \cdots \dashrightarrow X_i \dashrightarrow \cdots \dashrightarrow X_r=Y$$
over $Y$,
with $Y$ being $Q$-abelian.

By Theorem \ref{thm-GMMP}, this MMP is also $G$-equivaraint for some finite-index submonoid $G$ of $\SEnd(X)$.
Since $Y$ is $\Q$-abelian, any surjective endomorphism $g_r\in \SEnd(Y)$ is quasi-\'etale.
By \cite[Lemma 2.12]{Na-Zh} or \cite[Lemma 8.1 and Corollary 8.2]{CMZ},
$G_r$ lifts to a subsemigroup $G_A$ of $\SEnd(A) \le \End_{\variety}(A)$.
So (1) is proved.

(2) follows from \cite[Theorem 3.11 and Corollary 3.12]{MZ} and \cite[Lemmas 3.5 and 3.6]{Meng}. (3) follows from Lemma \ref{lem-amp-iamp}.
(4) and (5) follow directly from \cite[Theorem 1.10]{Meng}.

For (6), one direction is trivial and the case over $\C$ has been shown by Lemma \ref{lem-G-diag}.
Suppose $H_Y^*|_{\NS_{\Q}(Y)}$ is diagonalizable.
Then $H^*|_{\NS_{\mathbb{C}}(X)}$ is diagonalizable by Lemma \ref{lem-G-diag} and hence $H^*|_{\NS_{\Q}(X)}$ is commutative.
Let $h\in H$ and $\lambda$ be an eigenvalue of $h^*|_{\NS_{\C}(X)}$.
Then $\lambda$ is either an eigenvalue of $h_i^*|_{\NS_{\Q}(X_i)/\NS_{\Q}(X_{i+1})}$ or an eigenvalue of $h_r^*|_{\NS_{\Q}(Y)}$.
In particular, $\lambda\in \Q$.
So $h^*|_{\NS_{\Q}(X)}$ is diagonalizable for any $h\in G$.
By \cite[Section 15.4]{Hu}, $H^*|_{\NS_{\Q}(X)}$ is diagonalizable.
\end{proof}

\begin{proof}[Proof of Corollary \ref{main-cor-comp}]

For (I), (Ia) implies (Ib) by \cite[Theorem 3.11]{MZ} and \cite[Lemma 3.5]{Meng}.
Conversely, the diagonalizable case has been shown by Lemma \ref{lem-g-diag}.
Suppose $g$ is $q_g$-polarized, $h$ is $q_h$-polarized and $\tau_Y$ is $q$-polarized for some integers $q_g\ge 2, q_h\ge 2, q\ge 2$.
For each $i$, $g_i:=g|_{X_i}$ is $q_g$-polarized and $h_i:=h|_{X_i}$ is $q_h$-polarized by \cite[Lemma 3.10 and Theorem 3.11]{MZ}.
Since $\tau_Y^*|_{\NS_{\C}(Y)}$ is diagonalizable by \cite[Proposition 2.9]{MZ}, $\tau^*|_{\NS_{\C}(X)}$ is diagonalizable by Lemma \ref{lem-g-diag}.
Let $\lambda$ be an eigenvalue of $\tau^*|_{\NS_{\C}(X)}$.
Then $\lambda$ is either an eigenvalue of $\tau_i^*|_{\NS_{\C}(X_i)/\NS_{\C}(X_{i+1})}$ for some $i$ or an eigenvalue of $\tau_Y^*|_{\NS_{\C}(Y)}$.
Suppose $\lambda$ is an eigenvalue of $\tau_Y^*|_{\NS_{\C}(Y)}$ with $\dim(Y)>0$.
Note that $$\deg \tau_Y=q^{\dim(Y)}=(\deg h_Y)\cdot(\deg g_Y)=q_h^{\dim(Y)}\cdot q_g^{\dim(Y)}.$$
So $|\lambda|=q=q_h\cdot q_g$ by \cite[Lemma 2.1]{Na-Zh}.
Suppose $\lambda$ is an eigenvalue of $\tau_i^*|_{\NS_{\C}(X_i)/\NS_{\C}(X_{i+1})}$.
By Lemma \ref{lem-rel-int}, $$\tau_i^*|_{\NS_{\C}(X_i)/\NS_{\C}(X_{i+1})}=(h_i^*\circ g_i^*)|_{\NS_{\C}(X_i)/\NS_{\C}(X_{i+1})}=(q_h\,\id)\circ (q_g\,\id)=q\,\id.$$
Then $\lambda=q$.
Therefore, $\tau^*|_{\NS_{\C}(X)}$ is diagonalizable with all the eigenvalues being of the same modulus.
Applying \cite[Proposition 2.9]{MZ} and \cite[Lemma 2.3]{Na-Zh}, $\tau$ is $q$-polarized.

Suppose either $g$ or $h$ is int-amplified. Suppose $\tau_Y$ is int-amplified.
Let $\lambda$ be an eigenvalue of $\tau^*|_{\NS_{\C}(X)}$.
If $\lambda$ is an eigenvalue of $\tau_Y^*|_{\NS_{\C}(Y)}$, then $|\lambda|>1$ by \cite[Proposition 3.3]{Meng}.
Suppose $\lambda$ is an eigenvalue of $\tau_i^*|_{\NS_{\C}(X_i)/\NS_{\C}(X_{i+1})}$ for some $i$.
By Lemma \ref{lem-rel-int}, $g_i^*|_{\NS_{\C}(X_i)/\NS_{\C}(X_{i+1})}=a\,\id$ for some integer $a\ge 1$ and $h_i^*|_{\NS_{\C}(X_i)/\NS_{\C}(X_{i+1})}=b\,\id$ for some integer $b\ge 1$.
Since either $g$ or $h$ is int-amplified, either $a > 1$ or $b > 1$ by \cite[Proposition 3.3]{Meng}.
In particular, $\lambda=ab>1$.
By \cite[Proposition 3.3]{Meng} again, $\tau$ is int-amplified.

(II) follows from \cite[Corollary 3.12]{MZ} and \cite[Lemma 3.6]{Meng}.

(III) follows from Theorem \ref{main-thm-GMMP} by applying $H:=\{f,g\}$.
\end{proof}

\section{Proof of Theorems \ref{main-thm-rc} and \ref{main-thm-auto}}

In this section, we work over characteristic $0$.
We prove Theorems \ref{thm-q-pi-pt} and \ref{thm-auto} which include Theorems \ref{main-thm-rc} and \ref{main-thm-auto} as special cases.

\begin{definition}\label{def-q} Let $X$ be a normal projective variety.
\begin{itemize}
\item[(1)]
$q(X):=h^1(X,\mathcal{O}_X)=\dim H^1(X,\mathcal{O}_X)$ (the irregularity).
\item[(2)]
$\tilde{q}(X):=q(\tilde{X})$ with $\tilde{X}$ a smooth projective model of $X$.
\item[(3)]
$q^\natural(X):=\sup\{\tilde{q}(X')\,|\,X'\to X \text{ is finite surjective and \'etale in codimension one}\}$.
\item[(4)]
$\pi_1^{\alg}(X_{\reg})$ is the algebraic fundamental group of the smooth locus $X_{\reg}$ of $X$.
\end{itemize}
\end{definition}

\begin{theorem}\label{thm-q-pi-pt}
Let $X$ be a $\Q$-factorial klt projective variety admitting an int-amplified endomorphism $f$.
We use the notation $X=X_0 \dashrightarrow \cdots \dashrightarrow X_r=Y$ and the finite-index submonoid $G \le \SEnd(X)$ as in Theorem \ref{main-thm-GMMP}.
Suppose futher either $q^\natural(X)=0$ or $\pi_1^{\alg}(X_{\reg})$ is finite.
Then there is an integer $M \ge 1$ depending only on $X$ such that:
\begin{itemize}
\item[(1)] The $Y$ in Theorem \ref{main-thm-GMMP} is a point.
\item[(2)]
$G^*|_{\NS_{\Q}(X)}$ is a commutative diagonal
monoid with respect to a suitable $\Q$-basis $B$ of $\NS_{\Q}(X)$.
Further, for every $g$ in $G$,
the representation matrix $[g^*|_{\NS_{\Q}(X)}]_B$ relative to $B$,
is equal to $\diag[q_1, q_2, \dots]$ with integers $q_i \ge 1$.
\item[(3)]
$G \cap \Pol(X)$ is a subsemigroup of $G$, and consists exactly of those $g$ in $G$
such that $[g^*|_{\NS_{\Q}(X)}]_B = \diag[q, \dots, q]$ for some integer $q \ge 2$.
Further,
$$G \cap \Pol(X) \supseteq \langle \Pol(X)^{[M]} \rangle .$$
\item[(4)]
$G \cap \IAmp(X)$ is a subsemigroup of $G$, and consists exactly of those $g$ in $G$
such that $[g^*|_{\NS_{\Q}(X)}]_B = \diag[q_1, q_2, \dots]$ with integers $q_i \ge 2$.
Further,
$$G \, (G \cap \IAmp(X)) = G \cap \IAmp(X)  \supseteq  \langle \IAmp(X)^{[M]} \rangle ;$$
any $h$ in $\SEnd(X)$ has $(h^M)^* = (g_1^*)^{-1} g_2^*$ on $\NS_{\Q}(X)$
for some $g_i$ in $G \cap \IAmp(X)$.
\item[(5)]
We have $h^M\in G$ and that $h^*|_{\NS_{\C}(X)}$ is diagonalizable for every $h \in \SEnd(X)$.
\end{itemize}
\end{theorem}

\begin{proof}
We apply Theorem \ref{main-thm-GMMP} and use the notation there.
Note that $\pi:X\to Y$ is equi-dimensional and $\pi$ has irreducible fibres.
So (1) follows from \cite[Lemma 11.1]{CMZ} and the proof of \cite[Lemma 9.1]{MZ}.

The first half of (2) has been shown in Theorem \ref{main-thm-GMMP}.
For any $g\in G$, let $\lambda$ be an eigenvalue of $g^*|_{\NS_{\Q}(X)}$.
Then $\lambda$ is an eigenvalue of $g_j^*|_{\NS_{\Q}(X_j)/\NS_{\Q}(X_{j+1})}$ for some $j$.
By Lemma \ref{lem-rel-int}, $\lambda$ is a positive integer.
So (2) is proved.

By Corollary \ref{main-cor-comp}, $G\cap \Pol(X)$ and $G \cap \IAmp(X)$ are both semigroups.
For any $g\in G \cap \Pol(X)$, $[g^*|_{\NS_{\Q}(X)}]_B= \diag[q,\cdots, q]$ for some integer $q\ge 2$ by (2) and \cite[Lemma 2.1]{Na-Zh}.
For any $g\in G \cap \IAmp(X)$, $[g^*|_{\NS_{\Q}(X)}]_B = \diag[q_1, q_2, \dots]$ with integers $q_i \ge 2$ by (2) and \cite[Proposition 3.3]{Meng}.
Note that $\langle \SEnd(X)^{[M]} \rangle\subseteq G$ for some $M>0$.
So $G \cap \Pol(X) \supseteq \langle \Pol(X)^{[M]} \rangle$
and $G \cap \IAmp(X)  \supseteq  \langle \IAmp(X)^{[M]} \rangle$.
Since $G$ is a monoid,
$G \, (G \cap \IAmp(X)) = G \cap \IAmp(X)$ by Corollary \ref{main-cor-comp}.
For any $h\in \SEnd(X)$, $g_2:=h^M\circ f^M\in G \cap \IAmp(X)$.
Let $g_1:=f^M$, which is in $G \cap \IAmp(X)$.
Then $(h^M)^*=(g_1^*)^{-1}\circ g_2^*$ on $\NS_{\Q}(X)$.
So (3) and (4) are proved.
(5) is clear.
\end{proof}

\begin{theorem}\label{thm-auto}
Let $X$ be a $\Q$-factorial klt projective variety admitting an int-amplified endomorphism $f$.
Suppose futher either $q^\natural(X)=0$ or $\pi_1^{\alg}(X_{\reg})$ is finite.
Then we have:
\begin{itemize}
\item[(1)]
$\Aut(X)/\Aut_0(X)$ is a finite group.
Further, $\Aut_0(X)$ is a linear algebraic group.
\item[(2)]
Every amplified endomorphism of $X$ is int-amplified.
\item[(3)]
$X$ has no automorphism of positive entropy (nor amplified automorphism).
\end{itemize}
\end{theorem}
\begin{proof}
By Theorem \ref{thm-q-pi-pt}, we may run MMP 
$X=X_0 \dashrightarrow \cdots \dashrightarrow X_r=Y$
as in Theorem \ref{main-thm-GMMP}, with $Y$ being a point.
Moreover, for some $M>0$, $(g^M)^*|_{\NS_{\Q}(X)}=\id$ for any $g\in \Aut(X)$ since $g$ has inverse.
Let $H$ be an ample Cartier divisor of $X$ and let $H_g:=\sum_{i=0}^{M-1}(g^i)^*H$.
Then $H_g$ is ample and $g^*H_g\equiv H_g$.
Thus $[\Aut(X):\Aut_0(X)]<\infty$ (cf.~\cite[Theorem 1.2]{MZ}, \cite[Proposition 2.2]{Li}, \cite[Theorem 4.8]{Fuj}).

Let $X' \to X$ be an $\Aut(X)$-equivariant resolution of $X$.
By Theorems \ref{main-thm-GMMP} and \ref{thm-q-pi-pt}, $X$ and hence $X'$ are rationally connected.
So $X'$ has trivial $\Alb(X')$.
In particular, $\Aut_0(X')$ and hence $\Aut_0(X)$ are linear (cf.~\cite{Ma}).
Therefore, (1) is proved.

(2) follows from Lemma \ref{lem-amp-iamp}; see also Theorem \ref{main-thm-GMMP}.
(3) follows from (1) and (2); see also \cite[Lemma 3.10]{Meng}.
\end{proof}

\begin{proof}[Proof of Theorems \ref{main-thm-rc} and \ref{main-thm-auto}]
By \cite[Corollary 4.18]{De}, $\pi_1^{\alg}(X_{\reg})$ is trivial when $X$ is a rationally connected smooth projective variety.
Then Theorem \ref{main-thm-rc} follows from Theorem \ref{thm-q-pi-pt} and Theorem \ref{main-thm-auto} follows from Theorem \ref{thm-auto}.
\end{proof}

\end{document}